\documentclass[12pt]{article}
\usepackage[top=1in, bottom=1in, left=1in, right=1in]{geometry}

\usepackage{xcolor}
\usepackage{geometry}
\usepackage{amssymb,bm}
\usepackage{amsmath}
\usepackage{amsthm}
\usepackage{graphicx}
\usepackage{hyperref}
\usepackage{caption}
\usepackage{float} 
\hypersetup{pdfborder=0 0 0}

\newtheorem{theorem}{{\sc Theorem}}[section]

\newtheorem{lemma}[theorem]{{\sc Lemma}}
\newtheorem{corollary}[theorem]{Corollary}
\newtheorem{remark}[theorem]{Remark}

\newcommand\restr[2]{{ \left.\kern-\nulldelimiterspace #1 \vphantom{\big|} \right|_{#2}}}
\newcommand{\RR}{\mathbb{R}}
\newcommand{\ZZ}{\mathbb{Z}}
\newcommand{\CC}{\mathbb{C}}
\newcommand{\BS}{\mathbb{S}}

\newcommand{\n}{\noindent}
\newcommand{\comment}[1]{}
\renewcommand{\div}{\mathrm{div}}

\DeclareMathAlphabet{\pazocal}{OMS}{zplm}{m}{n}

\newcommand{\CT}{\pazocal{T}}

\newcommand{\PI}{\pazocal{I}}

\newcommand{\CS}{\pazocal{S}}

\usepackage{bbm}
\usepackage{subcaption}

\title{Scattering of plane waves}
\author{Narek Hovsepyan\footnote{Department of Mathematics, Rutgers University, New Brunswick, NJ, USA (nh507@math.rutgers.edu)} \quad and \quad Michael S. Vogelius\footnote{Department of Mathematics, Rutgers University, New Brunswick, NJ, USA (vogelius@math.rutgers.edu)}}
\date{}

\begin{document}
\maketitle

\begin{abstract}
We formulate a problem that can be viewed as a natural variation of the so-called Pompeiu or Schiffer problem in the context of scattering of plane waves for the Linear Helmholtz equation.  For the two dimensional version of this variation, we establish conditions on the wave numbers and incident directions, that ensure a non-vanishing scattered field.
\end{abstract}

\section{Introduction}
\setcounter{equation}{0}

The problem studied here can be viewed as a somewhat generalized variation of the Pompeiu (or Schiffer) problem and the analysis represents a continuation of our recent work \cite{HV}, extending this to more general domains, but restricting it to more specific incident waves -- namely, plane waves. We also introduce a material parameter $a$ in the principal part of the Helmholtz operator, which now reads $\nabla \cdot (a \nabla u)+k^2qu$. 

In \cite{HV}, we studied the case $a=1$ and considered (piecewise) analytic domains, whose boundary parametrization $x(t) = x_1(t)+ix_2(t)$, $t \in [-1,1]$ extends to a complex analytic function of $t$ in some domain $\CT \subset \CC$ and has a simple critical point $t_0 \in \CT$. Under a suitable topological assumption -- namely, that that there exists a contour joining the endpoints $\pm 1$, passing through $t_0$ and lying in $\CT$, such that $\Re x$ attains its maximum on this contour at $t_0$ -- we proved, using the method of steepest descent, that general incident waves with complex analytic extensions, which are nonvanishing at $\left( x_1(t_0), x_2(t_0) \right) \in \CC^2$, necessarily scatter. In particular, plane waves satisfy these conditions and therefore always scatter from such domains. We then applied this result to several examples, including elliptical regions, regular nonconvex regions, and regions with cusps.

We remark that, for $a=1$  Cakoni and Vogelius \cite{CV} showed that within the class of Lipschitz inhomogeneities (and constant $q\neq 1$), non-scattering can generically occur only if the inhomogeneity has a real analytic boundary. The Lipschitz assumption can be relaxed to allow for certain cusps, see \cite{SS21}. Their proof used tools from free boundary regularity theory and was inspired by the work of Williams \cite{w2}, who employed this theory to prove that, a domain lacking the Pompeiu property necessarily has a real analytic boundary. Plane waves fall within the generic category. Regularity results for a general anisotropic inhomogeneous medium (including the scalar, isotropic case $a \neq 1$) were obtained in \cite{CVX}. For a scalar $a \neq 1$ plane waves fall in the generic category for all points on the boundary, where the propagation direction is not tangential to the boundary.

Our analysis here is inspired by the work of Brown and Kahane \cite{BK}, who proved that a convex domain, whose maximal width is at least twice its minimal width, has the Pompeiu property (see also Remark~\ref{REM BK}). For the definition of the Pompeiu property see Remark~\ref{REM Pomp}. We refer to \cite{zalc} for a survey on the Pompeiu problem; see also the Appendix of \cite{HV}.

\section{Preliminaries}
\setcounter{equation}{0}

Throughout, boldface symbols denote vectors, while regular (non-bold) symbols are reserved for scalar quantities. Let us make the following assumptions:

\begin{enumerate}

\item[(H1)] $D \subset \RR^2$ is a bounded, simply connected Lipschitz domain.

\item[(H2)] Outside $D$ we have a homogeneous background medium with normalized material parameters $(1,1)$. Inside $D$, the material parameters are $(a, q)$, where $a, q > 0$ are constants with $(a,q) \neq (1,1)$. We introduce the refractive index

\begin{equation} \label{n}
n = \sqrt{\frac{q}{a}}.
\end{equation} 
\end{enumerate}

\n Here $a$ and $q$ represent the inverses of the mass density and bulk modulus of the medium, respectively. Consider a non-trivial incident wave at a fixed wave number $k>0$

\begin{equation} \label{u^in Helm}
\Delta u^{\text{in}} + k^2 u^{\text{in}} = 0 \qquad \qquad \text{in} \ \RR^2,
\end{equation}

\n impinging on the inhomogeneity $D$. Let $u^{\text{tr}}$ and $u^{\text{sc}}$ denote the corresponding transmitted and scattered waves, respectively. If $u^{\text{in}}$ is non-scattering, i.e., $u^{\text{sc}}=0$, then the total field outside $D$ equals $u^{\text{in}}$ and so

\begin{equation} \label{nonsc system}
\begin{cases}
a \Delta u^{\text{tr}} + k^2 q u^{\text{tr}} = 0 \hspace{1in} &\text{in} \ D, 
\\[.1in]
u^{\text{tr}} = u^{\text{in}}, \qquad a \, \partial_\nu u^{\text{tr}} = \partial_\nu u^{\text{in}} &\text{on} \ \Gamma = \partial D,
\end{cases}
\end{equation}

\n where $\nu$ denotes the unit outer normal to the boundary $\Gamma$ of the region $D$. From \eqref{nonsc system}, we obtain the following necessary condition for $u^{\text{in}}$ to be non-scattering (see Section~\ref{SECT LEM} for the proof):

\begin{lemma} \label{LEM phi}
If $u^{\text{in}}$ is non-scattering, then

\begin{equation} \label{integral phi general}
k^2 \left( \frac{1}{a} - n^2 \right) \int_D u^{\text{in}} \phi d\bm{x} - \left(\frac{1}{a} - 1 \right) \int_D \nabla u^{\text{in}} \cdot \nabla \phi d\bm{x} = 0
\end{equation}

\n for all test functions $\phi \in H^2(D)$ satisfying $\Delta \phi + k^2 n^2 \phi = 0$ in $D$.
\end{lemma}

Our goal is to derive a contradiction to the above equation for incident plane waves. Let us therefore consider a fixed incident plane wave in the direction of the unit vector $\bm{\eta}$:

\begin{equation} \label{u^in}
u^{\text{in}}(\bm{x}) = e^{ik \bm{\eta} \cdot \bm{x}}, \qquad \qquad \bm{\eta} \in \BS = \{\bm{\eta} \in \RR^2: \bm{\eta} \cdot \bm{\eta} = 1\}.
\end{equation} 

\n The main idea is to also take plane wave solutions as test functions, but with a complex wave vector. Namely, $\phi(\bm{x}) = e^{i kn \bm{\xi} \cdot \bm{x}}$, where $\cdot$ denotes the inner product of $\RR^2$ and $\bm{\xi} \in \CS$, where $\CS$ denotes the complexification of the unit circle $\BS$, {\it i.e.},

\begin{equation} \label{CS}
\CS = \left\{ \bm{\xi} \in \CC^2 : \bm{\xi} \cdot \bm{\xi} = 1 \right\}.
\end{equation}

\n The restriction imposed on $\bm{\xi}$ ensures that $\phi$ satisfies the required PDE. For this choice of test function, $\nabla u^{\text{in}} \cdot \nabla \phi = -k^2 n (\bm{\xi} \cdot \bm{\eta})  u^{\text{in}} \phi$ and \eqref{integral phi general} simplifies to

\begin{equation} \label{I = 0}
\PI(\bm{\xi}) : = C(\bm{\xi}) I(\bm{\xi}) = 0 \qquad \qquad \forall \ \bm{\xi} \in \CS,
\end{equation}

\n where

\begin{equation} \label{C I def}
C(\bm{\xi}) = \frac{1}{a} - n^2 + \left( \frac{1}{a} - 1 \right) n \, \bm{\xi}\cdot \bm{\eta}  \qquad \text{and} \qquad  I(\bm{\xi}) = \int_D e^{ik(\bm{\eta} + n \bm{\xi}) \cdot \bm{x}} d\bm{x}.  
\end{equation}

\begin{remark} \label{REM Pomp}
\normalfont
To highlight the connection between the failure of the Pompeiu property and the non-scattering phenomenon, we recall the definition of the Pompeiu property. Rather than using the original formulation, we invoke the celebrated result of Brown, Schreiber, and Taylor \cite{BST73} (see also \cite{bern80}), which gives an equivalent definition: a bounded, simply connected Lipschitz domain $D$ fails to have the Pompeiu property if and only if there exists a constant $\rho>0$ such that

\begin{equation*}
\int_D e^{i \rho \bm{\xi} \cdot \bm{x}} d \bm{x} =0 \qquad \qquad \forall \ \bm{\xi} \in \CS.
\end{equation*}

\end{remark}

\vspace{.1in}

To conclude that the plane wave $e^{i k \bm{\eta} \cdot \bm{x}}$ scatters, we look for a test vector $\bm{\xi} \in \CS$, such that $\PI(\bm{\xi}) \neq 0$. In Theorem~\ref{THM n<1} we prove that such a choice is always possible when $n \leq 1$, which implies that a plane wave always scatters, independently of $k, \bm{\eta}$ and the geometry of $D$. In fact, when $n<1$ we use a complex-valued admissible $\bm{\xi}$, while in the case $n=1$ this choice becomes inadmissible. Instead, we use real vectors $\bm{\xi}$ to derive a set of inequalities involving $k, \bm{\eta}$, and the directional widths of $D$, that guarantee scattering. By varying the test vector $\bm{\xi}$ these inequalities are shown to cover the entire spectrum $k \in (0, \infty)$. The case $n > 1$ is more delicate and is treated in Theorem~\ref{THM n>1}, where inequalities involving $n, k$, $\bm{\eta}$ and directional widths of $D$ are derived to guarantee the scattering of plane waves. Unlike the situation for $n \leq 1$, these inequalities do not cover the full spectrum. They guarantee scattering on intervals of the form 

\begin{equation*}
k \in \left( 0, \tfrac{k_+}{2} \right],
\end{equation*}

\n where $k_+$ depends explicitly on the directional widths of $D$, $\bm{\eta}$ and $n$. For strictly convex regions $D$ stronger results can be obtained. In Theorem~\ref{THM convex}, we show that scattering is also guaranteed on intervals of the form

\begin{equation*}
k \in [m k_-, m k_+], \qquad \quad m=1,2,...,
\end{equation*}

\n where $k_-$ is given by an explicit formula. These intervals drift towards infinity and grow larger as $m$ increases. Under the assumption that $k_+ \geq 2 k_-$, a condition depending only on $n$, $\bm{\eta}$, and the directional widths of $D$, the union of all these intervals covers the full spectrum $k \in (0, \infty)$. Thus, under this assumption,  scattering is guaranteed for every $k$,  as asserted in Corollary~\ref{CORO h1>2h0}. If $na \leq 1$, a discrete sequence of $k$'s must be excluded; see Remark~\ref{REM na<1}. Finally, in Corollary~\ref{CORO two examples} we apply this criterion to two classes of strictly convex domains. Restricting to the case $na > 1$ (which includes the case $a=1$, $n=\sqrt{q}>1$) we show that:

\begin{enumerate}
\item[(1)] If $D$ has constant width, then all plane waves scatter provided $n \leq 3$.

\item[(2)] If the maximal width of $D$ is at least twice its minimal width ({\it i.e.}, $D$ is at least twice as wide in one direction as in some other direction) then all plane waves scatter.
\end{enumerate}

\n Finally, in Section~\ref{SECT slab} we study a simple model in which $D$ is an unbounded domain -- namely, a vertical slab -- and consider an incident plane wave at normal incidence. This reduces the problem to a one-dimensional setting. In Lemma~\ref{LEM slab}, we show that this incident plane wave is non-scattering for appropriate choices of $k, n,$ and the slab thickness. In Remark~\ref{REM slab and disk}, we discuss the analogy between this and the non-scattering of particular Herglotz waves by a disk. Our motivation for considering the infinite slab comes from the so-called Salisbury screen \cite{Sal, knott}, one of the earliest models exhibiting non-scattering of incident plane waves. While this is a classical model well studied in engineering and physics, it has received less attention in the mathematics literature.

\section{Main Results}
\setcounter{equation}{0}

The proofs of all main results are presented in Section~\ref{SECT Proofs}.

\subsection{General inhomogeneities}

\begin{theorem} \label{THM n<1}
Assume (H1) and (H2). If $n \leq 1$, then every incident plane wave \eqref{u^in} is scattered by the inhomogeneity $D$.
\end{theorem}

We note that the case $n<1$ corresponds to waves traveling faster inside the medium $D$ than in the background, while the case $n=1$ corresponds to equal wave speeds inside and outside of $D$.

\begin{remark} \label{REM n=1}
\normalfont The fact that plane waves always scatter when $n\leq 1$ is particularly interesting, given that there exist other examples of non-scattering incident waves in this case. For the disk (with $a=1$ and $q\neq 1$) there is an infinite sequence of wave numbers at which one can find non-scattering Herglotz waves \cite{CK,CV}. For the unit square (with $a=q\neq 1$) the incident wave $u^{\text{in}}(\bm{x}) = \cos m \pi x_1 \cos l \pi x_2$, which solves the Helmholtz equation with $k^2 = \pi^2 (m^2 + l^2)$, is non-scattering for any $m, l \in \ZZ$ \cite{CVX}.  
\end{remark}

We now turn to the case $n >1$. Before stating our results, we introduce some important quantities. For $\bm{\lambda} \in \BS$, let $w(\bm{\lambda})$ denote the width of $D$ in the direction $\bm{\lambda}$, defined as the length of the orthogonal projection of $D$ onto the line spanned by $\bm{\lambda}$. That is,

\begin{equation*}
w(\bm{\lambda}) = \sup_{\bm{x} \in D} \bm{\lambda} \cdot \bm{x} - \inf_{\bm{x} \in D} \bm{\lambda} \cdot \bm{x}.
\end{equation*}

\n We also introduce the positive increasing function

\begin{equation} \label{M}
M(r) = r + \sqrt{n^2 - 1 + r^2}.
\end{equation}

\begin{theorem} \label{THM n>1}
Assume (H1), (H2) and $n>1$. The incident plane wave \eqref{u^in} in direction $\bm{\eta} \in \BS$ is scattered by the inhomogeneity $D$ provided 

\begin{itemize}

\item[(i)] $n a > 1$ and

\begin{equation*}
k \leq \frac{\pi}{\displaystyle w(\bm{\lambda}) M(\bm{\lambda} \cdot \bm{\eta})} \qquad \qquad \text{for some} \ \bm{\lambda} \in \BS, \ or
\end{equation*}

\item[(ii)] $n a \leq 1$ and 

\begin{equation*}
k \leq \frac{\pi}{\displaystyle w(\bm{\lambda}) M(\bm{\lambda} \cdot \bm{\eta})} \qquad \qquad \text{for some} \ \bm{\lambda} \in \BS\backslash \{\bm{\lambda}_\pm\},
\end{equation*}

where $\bm{\lambda}_\pm$ are the unit vectors that satisfy the equation

\begin{equation} \label{lambda eta = r}
\bm{\lambda} \cdot \bm{\eta} = \frac{a \sqrt{n^2 - 1}}{\sqrt{1-a^2}} =: r_0.
\end{equation}

\n In other words, $\bm{\lambda}_\pm$ are the vectors forming angles $\pm \arccos r_0$ with $\bm{\eta}$.

\end{itemize}

\end{theorem}

Note that any admissible direction $\bm{\lambda}$ can be chosen to obtain an inequality ensuring the scattering of the corresponding plane wave. In part $(i)$, every direction $\bm{\lambda}$ is admissible, while in part $(ii)$ all but two exceptional directions are admissible. If we rewrite $r_0$ in the slightly different form

\begin{equation} \label{r_0 = 1 -}
r_0 =  \frac{\sqrt{(na)^2 - a^2}}{\sqrt{1-a^2}}~,
\end{equation}

\n it is clear that $r_0 \in (0, 1]$ if $n>1$ and $na \leq 1$. 

There are two particularly useful directions in our estimates: $\bm{\lambda} = - \bm{\eta}$ and $\bm{\lambda} = \bm{\eta}^\perp$, where $\bm{\eta}^\perp$ denotes a unit vector perpendicular to $\bm{\eta}$ (the specific choice between the two possible orthogonal directions is irrelevant). Since $r_0$ cannot take the values $-1$ or $0$, these directions are also admissible in part $(ii)$. Therefore, for $n>1$, scattering is guaranteed provided

\begin{equation*}
k \leq \frac{\pi}{w(\bm{\eta}) M(-1)}, \qquad \text{or} \qquad k \leq \frac{\pi}{w(\bm{\eta}^\perp) M(0)}.
\end{equation*}

\n Simplifying the above values of the function $M$, we arrive at

\begin{corollary}
Assume (H1), (H2) and $n>1$. The plane wave $u^{\text{in}}$ is scattered by $D$ provided that

\begin{equation} \label{k eta and perp}
k \leq \max \left\{ \frac{\pi}{w(\bm{\eta}) \left( n - 1 \right)}; \frac{\pi}{w(\bm{\eta}^\perp) \sqrt{n^2 - 1} } \right\}.
\end{equation}

\end{corollary}

In particular, the above estimate shows continuity between the cases $n \leq 1$ and $n>1$. Indeed, as $n \to 1^+$ the guaranteed scattering extends to the full frequency spectrum $k \in (0, \infty)$. It also admits a useful physical interpretation. Fix an inhomogeneity $D$ with material parameter $n>1$. Then plane waves with sufficiently small $k$ (i.e. large wavelength) always scatter from $D$. If the inhomogeneity is much smaller than the wavelength, the incident wave ``sees" the inhomogeneity only weakly and is not expected to interact strongly with it. The estimate shows that one cannot achieve non-scattering by relying on such weak interactions, and that stronger wave–matter interactions are needed to even hope to achieve non-scattering. This in itself is not unexpected: in many interesting cases, for instance when $a=1$, or when $a-1$ and $q-1$ have opposite sign,  it is known that the real transmission eigenvalue spectrum is bounded away from zero (see Theorem 4.15  and Theorem 4.42 of \cite{CCH}). Since being a real transmission eigenvalue is a nececessary condition for a wave number to have an associated non-scattered incident wave, it follows that any incident wave, in particular any plane wave, must scatter for sufficiently small wave numbers in these cases. What is entirely novel about our estimates is their very explicit dependence on the geometry of the domain and the direction of the incident plane wave.

Strong interactions correspond to small wavelengths (i.e., large $k$), and this case is treated in Section~\ref{SEC convex} for strictly convex inhomogeneities. For such domains we obtain an infinite set of increasing frequency intervals with guarantied scattering. Furthermore, if the maximal width of the domain is more than twice the minimal width, then the union of these intervals covers all of $\mathbb{R}^+$, and so we have scattering of plane waves for all wave numbers $k$.

Let us next restate Thereom~\ref{THM n>1} in a slightly different form. The largest interval of $k$ with guaranteed scattering is obtained by minimizing the denominator, so let

\begin{equation} \label{h_0 def}
h_0(\bm{\eta}) = \min_{\bm{\lambda} \in \BS} \left\{ w(\bm{\lambda}) M(\bm{\lambda} \cdot \bm{\eta}) \right\}.
\end{equation}

\n When $na > 1$, scattering is guaranteed for

\begin{equation} \label{k < h_0}
k \leq \frac{\pi}{h_0(\bm{\eta})}.
\end{equation}

\n Since we are minimizing a continuous function over a compact set, the minimum $h_0(\bm{\eta})$ is attained at some point $\bm{\lambda}_0 \in \BS$. If $\bm{\lambda}_0 \neq \bm{\lambda}_\pm$, then for $na \leq 1$ we again have guaranteed scattering in the interval \eqref{k < h_0}. Otherwise, if $\bm{\lambda}_0 = \bm{\lambda}_\pm$, then $h_0(\bm{\eta}) = w (\bm{\lambda}_\pm) M(r_0)$ and the endpoint $k = \pi / h_0(\bm{\eta})$ must be excluded from the interval \eqref{k < h_0} in order to ensure scattering.

\begin{remark}
\normalfont
We do not believe the bounds established here are necessarily optimal. In particular for strictly convex inhomogeneities we are inclined to believe that plane waves always scatter. This belief is in part supported by the enhanced results for such domains, which we describe in detail in the following section.
\end{remark}

\begin{remark}[Sound soft/hard obstacles] \mbox{}
\normalfont

\n The limit $a \to \infty$ (or $n \to 0$) includes scattering by a Dirichlet (sound-soft) obstacle as a particular case. Theorem~\ref{THM n<1} applies here and implies that plane waves always scatter from a sound-soft obstacle. The limit $a \to 0$ (or $n \to \infty$) corresponds to scattering by a Neumann (sound-hard) obstacle. Part $(i)$ of Theorem~\ref{THM n>1} applies in this case; however, the estimate on $k$ degenerates to 0. We remark that plane waves always scatter from both sound-soft and sound-hard obstacles -- this is immediate and does not require the above results. Indeed, if $u^{\text{in}}$ were non-scattering then for a sound-soft obstacle $u^{\text{in}} = 0$ on $\Gamma$, and for a sound-hard obstacle $\partial_\nu u^{\text{in}} = 0$ on $\Gamma$. These conditions cannot be satisfied by the plane wave \eqref{u^in}.
\end{remark}

\subsection{Strictly convex inhomogeneities} \label{SEC convex}

Along with \eqref{h_0 def}, introduce

\begin{equation} \label{h_1 def}
h_1(\bm{\eta}) = \max_{\bm{\lambda} \in \BS} \left\{ w(\bm{\lambda}) M(\bm{\lambda} \cdot \bm{\eta}) \right\}.
\end{equation}

\begin{theorem} \label{THM convex}
Assume (H1), (H2), and that $D$ is strictly convex. Suppose $n>1$ and $na > 1$. The incident plane wave \eqref{u^in} in direction $\bm{\eta} \in \BS$ is scattered by the inhomogeneity $D$ provided that

\begin{equation} \label{k in union}
k \in \bigcup_{m=1}^\infty \left[ \frac{2\pi m}{h_1(\bm{\eta})}, \frac{2\pi m}{h_0(\bm{\eta})} \right].
\end{equation}

\end{theorem}

\vspace{.1in}

\begin{remark} \label{REM na<1}
\normalfont The above result also holds for $na \leq 1$, except that the wave numbers

\begin{equation*}
k = \frac{2\pi m}{w(\bm{\lambda}_\pm) M(r_0)}, \qquad \qquad m=1,2,...
\end{equation*}

\n (potentially) must be removed from the union \eqref{k in union}. Recall that $\bm{\lambda}_\pm$ are the exceptional directions defined in Theorem~\ref{THM n>1}. 
\end{remark}

\begin{remark} \normalfont
Note that $h_0(\bm{\eta}) \neq h_1(\bm{\eta})$ for every $\bm{\eta}$. Indeed, if equality held, then the function $w(\bm{\lambda}) M(\bm{\lambda} \cdot \bm{\eta})$ would have to be constant for all $\bm{\lambda}$, but its values at $\bm{\lambda} = \pm \bm{\eta}$ are always different. This implies that the intervals in \eqref{k in union} eventually start overlapping as $m$ grows, and thus cover an interval of the form $[k_0, \infty)$. Consequently, the possible non-scattering wave numbers $k>0$ cannot accumulate at $\infty$. Since it is known that the non-scattering wave numbers have no finite accumulation point \cite{CCH,VX}, it follows under the assumptions of Theorem \ref{THM convex} that the set of non-scattering wave numbers is finite (possibly empty) for any fixed incident direction $\bm{\eta}$. This finiteness result (in the case $a=1$) was previously established in \cite{VX} using asymptotic methods.      
\end{remark}

Since the intervals in \eqref{k in union} expand as $m$ increases, Theorem~\ref{THM convex} in particular applies to sufficiently large wave numbers $k$. Let us be more specific and draw a simple conclusion from this result in terms of more explicit and easier-to-compute quantities, namely the maximal and minimal widths of $D$:

\begin{equation} \label{w* and w_*}
w^* = \max_{\bm{\lambda} \in \BS} w(\bm{\lambda}), \qquad \qquad w_* = \min_{\bm{\lambda} \in \BS} w(\bm{\lambda}).
\end{equation}

\begin{corollary}
Assume the hypotheses of Theorem~\ref{THM convex}, and suppose that $w^* (n-1) > w_* (n+1)$. Then the incident plane wave \eqref{u^in} is scattered by $D$ provided 

\begin{equation} \label{k >}
k \geq \frac{2\pi}{w^* (n-1) - w_* (n+1)},
\end{equation}

\n regardless of the direction $\bm{\eta} \in \BS$.
\end{corollary}

\begin{proof}
Since $M$ is an increasing function, we apply the basic estimates

\begin{equation*}
h_1(\bm{\eta}) \geq w^* M(-1) = w^* (n-1), \qquad \qquad h_0(\bm{\eta}) \leq w_* M(1) = w_* (n+1),
\end{equation*}

\n which hold for all $\bm{\eta} \in \BS$. Theorem~\ref{THM convex} implies that scattering is guaranteed if there exists some $m = 1,2,...$ such that

\begin{equation*}
k \in \left[ \frac{2\pi m}{w^* (n-1)}, \frac{2\pi m}{w_* (n+1)} \right]
\qquad \Longleftrightarrow \qquad m \in \left[ \frac{k w_*}{2\pi}(n+1), \frac{kw^* }{2\pi}  (n-1) \right]
\end{equation*}

\n Clearly, an integer $m$ exists in the above interval, provided its length is at least 1, which is precisely the condition \eqref{k >}.

\end{proof}

Theorem~\ref{THM convex} can be combined with part $(i)$ of Theorem~\ref{THM n>1}, and it is easy to see that when $h_1(\bm{\eta}) \geq 2 h_0(\bm{\eta})$, the consecutive intervals starting from \eqref{k < h_0} and continuing with those in \eqref{k in union} overlap and cover the entire frequency spectrum $k \in (0, \infty)$. Therefore, we obtain

\begin{corollary} \label{CORO h1>2h0}
Assume the hypotheses of Theorem~\ref{THM convex}, and suppose the direction vector $\bm{\eta}$ satisfies

\begin{equation} \label{h_1 > 2 h_0}
h_1(\bm{\eta}) \geq 2 h_0(\bm{\eta}),
\end{equation}

\n then the incident wave \eqref{u^in} in direction $\bm{\eta}$ is scattered by $D$, regardless of the value of $k$.
\end{corollary}

The condition \eqref{h_1 > 2 h_0} is generally not straightforward to apply in a specific scattering problem. However, it has two interesting consequences for two different types of (strictly convex) domains, which are easy to establish and for which \eqref{h_1 > 2 h_0} is satisfied for all incidence directions $\bm{\eta}$. 

\begin{corollary} \label{CORO two examples}
Assume the hypotheses of Theorem~\ref{THM convex}.

\begin{enumerate}
\item[$(i)$] Assume that $D$ has constant width\footnote{Domains with constant width include disks, but there are many other constant width domains with analytic boundaries, see e.g. \cite{JPF}. For a comprehensive, up-to-date discussion of constant width domains, see \cite{HM}. Incidentally, a convex domain of constant width is automatically strictly convex. }, i.e., $w(\bm{\lambda}) = const$ for all $\bm{\lambda} \in \BS$, and that $n \leq 3$. Then every incident plane wave \eqref{u^in} is scattered by $D$.

\item[$(ii)$] Assume $D$ satisfies $w^* \geq 2 w_*$, i.e., the maximal width of $D$ is at least twice its minimal width. Then every incident plane wave \eqref{u^in} is scattered by $D$.

\end{enumerate}

\end{corollary}

\begin{proof}
$(i)$ Write $w(\bm{\lambda}) = w$. Since $M$ is an increasing function, we have

\begin{equation*}
h_1(\bm{\eta}) = w \max_{\bm{\lambda} \in \BS} \left\{ M(\bm{\lambda} \cdot \bm{\eta}) \right\} = w M(1).
\end{equation*}

\n In other words, $h_1$ is independent of $\bm{\eta}$. An analogous formula holds for $h_0(\bm{\eta})$, so that the inequality \eqref{h_1 > 2 h_0} reduces to $M(1) \geq 2 M(-1)$. Or equivalently, using the definition \eqref{M}, this becomes $1+ n \geq 2 (-1 + n)$, i.e. $n \leq 3$. 

\vspace{.1in}

$(ii)$ Assume that the maximum and minimum in \eqref{w* and w_*} are attained at $\bm{\lambda}^*$ and $\bm{\lambda}_*$, respectively. Since $w(\bm{\lambda}) = w(-\bm{\lambda})$, we can then estimate

\begin{equation*}
h_1(\bm{\eta}) \geq w(\bm{\lambda}^*) M(\pm \bm{\lambda}^* \cdot \bm{\eta}) \geq 2 w(\bm{\lambda}_*) M(\pm \bm{\lambda}^* \cdot \bm{\eta}).
\end{equation*}

\n Suppose that $\bm{\lambda}^* \cdot \bm{\eta} \geq \bm{\lambda}_* \cdot \bm{\eta}$. Since $M$ is an increasing function we obtain

\begin{equation*}
h_1(\bm{\eta}) \geq 2 w(\bm{\lambda}_*) M(\bm{\lambda}_* \cdot \bm{\eta}) \geq 2 h_0(\bm{\eta}).
\end{equation*}

\n On the other hand, if $-\bm{\lambda}^* \cdot \bm{\eta} \geq -\bm{\lambda}_* \cdot \bm{\eta}$, we obtain the same estimate using $-\bm{\lambda}_*$. Hence, \eqref{h_1 > 2 h_0} holds in either case, regardless of $\bm{\eta}$, and the proof is complete.  

\end{proof}

\begin{remark}
\normalfont When $a=1$, our results from \cite{HV} imply that plane waves always scatter from any elliptical region $D$, including the case when $D$ is a disk, regardless of the value of $n=\sqrt{q}\neq 1$, and without any condition on maximal vs. minimal width.
\end{remark}

\begin{remark} \label{REM BK}
\normalfont Corollary~\ref{CORO h1>2h0} or part $(ii)$ of Corollary~\ref{CORO two examples} can be viewed as an extension of the result of Brown and Kahane \cite{BK} to the scattering context. Their result states that a strictly convex set with $w^* \geq 2 w_*$ has the Pompeiu property. Formally, the Pompeiu context corresponds to setting $u^{\text{in}} = 1$, i.e. $\bm{\eta} = 0$ in the integral $I(\bm{\xi})$ \eqref{C I def}. In that case, we can freely rotate $D$, since rotations can be absorbed into the test vector parameter $\bm{\xi}$, which varies in a rotationally invariant set $\CS$. This rotational invariance simplifies the analysis and leads to their conclusion provided $h_1 \geq 2h_0$, where these functions no longer contain the term $M(\bm{\lambda} \cdot \bm{\eta})$. In other words, $h_1$ and $h_0$ reduce to $w^*$ and $w_*$, respectively. Moreover, in the Pompeiu context the term $C(\bm{\xi})$ \eqref{C I def} is not present.
\end{remark}

\subsection{Infinite slab: non-scattering of plane waves} \label{SECT slab}

Let the inhomogeneity be an infinite vertical slab of thickness $w$, i.e., $D = \{(x,y): 0< x < w\}$, with material parameters $a=1$ and $q>0$ constant.  Outside the medium, the parameters are $(1,1)$. As before $n = \sqrt{q}$. We consider an incident plane wave at normal incidence, traveling along the $x$-axis:

\begin{equation} \label{u^in Salisbury}
u^{\text{in}}(x) = e^{-ikx}.
\end{equation}

\n Consequently, the scattered and transmitted waves also propagate along the $x$-axis, making the problem effectively one-dimensional.

\begin{lemma} \label{LEM slab}
Assuming the setting introduced above, the incident plane wave \eqref{u^in Salisbury} is non-scattering if and only if

\begin{enumerate}
\item[$\bullet$] $\displaystyle k w = 2\pi m$ \quad and \quad $\displaystyle n = \frac{l}{m}$ \quad for some $m,l = 1,2,...$ ($m \neq l$), or

\vspace{.1in}

\item[$\bullet$] $\displaystyle k w = \pi + 2\pi m$ \quad and \quad $\displaystyle n = \frac{1+2l}{1+2m}$ \quad for some $m,l = 0,1,2,...$ ($m\neq l$)
\end{enumerate}

\end{lemma}

\begin{proof}
With  $u = u^{\text{tr}} - u^{\text{in}}$, the non-scattering problem \eqref{nonsc system} reduces to

\begin{equation*}
\begin{cases}
u'' + k^2 n^2 u = k^2 (1-n^2) u^{\text{in}} \qquad \qquad \text{in} \ (0,w)
\\
u(0)=u(w)=0
\\
u'(0)=u'(w)=0.
\end{cases}
\end{equation*}

\n This is a simple ODE whose solution is $u(x) = -e^{-ikx} + c_1 \sin (knx) + c_2 \cos(knx)$, where $c_1, c_2 \in \CC$. The boundary conditions $u(0)=u'(0)=0$ imply that $c_2 = 1$ and $c_1 = -i/n$. The remaining two boundary conditions can be reduced to

\begin{equation*}
\sin (knw) = 0 \qquad \text{and} \qquad e^{-ikw} = \cos(knw),
\end{equation*}

\n which concludes the proof.

\end{proof}

\begin{remark}
\normalfont
In electromagnetism, the classical Salisbury screen consists of the slab $D$, along with a resistive sheet at $x=0$ with resistivity $R>0$ and a perfect conductor at $x=w$. Both can be modeled using appropriate boundary conditions. Similar to Lemma~\ref{LEM slab}, one can show that a suitable choice of the parameters $k, w$ and $R$ leads to non-scattering of the incident plane wave \eqref{u^in Salisbury}.
\end{remark}

\begin{remark} \label{REM slab and disk}
\normalfont
The above model can be viewed as the plane-wave analogue of the non-scattering of a particular Herglotz wave for the disk. Let us describe this analogy in more detail. Herglotz waves are continuous superpositions of plane waves; a particular example is the radial incident Herglotz wave $u^{\text{in}}(\bm{x}) = J_0(k|\bm{x}|)$, where $J_0$ is the Bessel function of order $0$. Let $D=D_w(0)$ be the disk of radius $w>0$ centered at the origin, and assume the same material parameters as in the slab case above. Then this incident wave is non-scattering for $D$ if and only if \cite{CK}

\begin{equation*}
J_0'(kw) J_0(knw) - n J_0(kw) J_0'(knw) = 0.
\end{equation*}

\n Due to the radial symmetry, this non-scattering problem also reduces to a one-dimensional problem, and the above equation then follows easily. Note the analogy: in both cases, the geometry of the inhomogeneity perfectly matches the wavefronts of the corresponding non-scattering incident wave. 
\end{remark}

\section{The Proofs} \label{SECT Proofs}
\setcounter{equation}{0}

In the following subsections, we present the proofs of Theorems~\ref{THM n<1}, \ref{THM n>1}, \ref{THM convex}, and Lemma~\ref{LEM phi}. As already discussed, the proofs proceed by contradiction: we assume that the incident plane wave is non-scattering. Then, by Lemma~\ref{LEM phi} (see \eqref{I = 0}), this assumption implies that $\PI(\bm{\xi}) = 0$ for all $\bm{\xi} =(\xi_1, \xi_2) \in \CC^2$ with $\bm{\xi} \cdot \bm{\xi} = 1$, i.e.,

\begin{equation} \label{xi^2=1}
\xi_1^2 + \xi_2^2 = 1.
\end{equation} 

\n The desired contradiction is obtained by finding $\bm{\xi}$ for which $\PI(\bm{\xi}) \neq 0$.

\subsection{Proofs of Theorems~\ref{THM n<1} and \ref{THM n>1}}

Recall that $\PI(\bm{\xi}) = C(\bm{\xi}) I(\bm{\xi})$, where $C$ and $I$ are given by \eqref{C I def}. For convenience, we restate these definitions:

\begin{equation*}
C(\bm{\xi}) = \frac{1}{a} - n^2 + \left( \frac{1}{a} - 1 \right) n \, \bm{\xi}\cdot \bm{\eta}  \qquad \text{and} \qquad  I(\bm{\xi}) = \int_D e^{ik(\bm{\eta} + n \bm{\xi}) \cdot \bm{x}} d\bm{x}.  
\end{equation*}

\vspace{.1in}

\textbf{The case $n<1$.} To simplify the analysis, we rotate the coordinate axes so that the incident direction is $\bm{\eta} = (1,0)$. The idea is simple: if both components of $\bm{\eta} + n \bm{\xi}$ are in $i\RR$, then $I(\bm{\xi})$ is clearly nonzero. So we attempt to select $\bm{\xi}$ so that

\begin{equation*}
1 + n \xi_1 \in i \RR \qquad \text{and} \qquad n \xi_2 \in i\RR.
\end{equation*}

\n Writing $\xi_2 = i y$ with $y \in \RR$, we conclude from \eqref{xi^2=1} that 

\begin{equation} \label{xi_1 lambda}
\xi_1 = \pm \sqrt{1+y^2},   
\end{equation}

\n which are real numbers. Therefore, $1+n\xi_1 \in i\RR$ only when it vanishes, i.e., $\xi_1 = -1/n$. This is consistent with \eqref{xi_1 lambda} provided $n<1$, and we choose the minus sign in \eqref{xi_1 lambda} with $y^2 = \frac{1}{n^2} - 1$.  With this choice of $\bm{\xi}$ we have $C(\bm{\xi}) = 1-n^2 \neq 0$, and so $\PI(\bm{\xi}) \neq 0$. This concludes the proof of Theorem~\ref{THM n<1} for the case $n<1$. 

\vspace{.1in}

\textbf{The case $n>1.$} The preceding method does not produce an admissible $\bm{\xi}$. Instead we consider real-valued vectors $\bm{\xi}$, so that $\bm{\xi} \in \BS$ is a true direction vector. 

Let us start by analyzing the integral $I(\bm{\xi})$. An important observation is that as $\bm{\xi}$ varies over all possible directions, so does the vector $k(\bm{\eta} + n \bm{\xi})$, since $n>1$. Indeed, note that $\bm{\eta} + n \bm{\xi}$ traces a circle centered at $\bm{\eta}$ with radius $n$. Because $n>1$, this circle contains the origin in its interior and thus spans all directions. In other words, given an arbitrary direction vector $\bm{\lambda} \in \BS$ we can find a vector $\bm{\xi}\in \BS$ and a constant $R>0$ such that

\begin{equation}\label{lambda xi}
k(\bm{\eta} + n \bm{\xi}) = R \bm{\lambda}.
\end{equation}

\n More precisely

\begin{equation*}
\bm{\xi} = \frac{1}{n} \left( \frac{R}{k} \bm{\lambda} - \bm{\eta} \right),
\end{equation*}

\n and since this must be a unit vector, $R$ must satisfy

\begin{equation} \label{R quadratic}
\frac{R^2}{k^2} - 2 \frac{R}{k} \bm{\lambda} \cdot \bm{\eta} + 1  - n^2 = 0.
\end{equation}

\n This is a quadratic equation with positive discriminant (as $n>1$). It has two roots, one positive and one negative. Since we are interested in the positive root, it follows that

\begin{equation*}
R = R(\bm{\lambda}) = k \left( \bm{\lambda} \cdot \bm{\eta} + \sqrt{n^2 - 1 + (\bm{\lambda} \cdot \bm{\eta})^2} \right).
\end{equation*}

\n In terms of the function $M$ defined in \eqref{M}, we have

\begin{equation} \label{R k M}
R = k M (\bm{\lambda} \cdot \bm{\eta}).
\end{equation}

\n Let us change the test vector variable from $\bm{\xi}$ to $\bm{\lambda}$, using \eqref{lambda xi} and \eqref{R k M}. In that case,

\begin{equation*}
I(\bm{\xi}) = \int_D e^{i R \bm{\lambda} \cdot \bm{x}} d \bm{x} =: \tilde{I} (\bm{\lambda}).
\end{equation*}

\n Using Fubini's theorem, we now aim to reduce the above integral to a one-dimensional integral. Note that the integrand is constant along lines perpendicular to $\bm{\lambda}$. Therefore, we slice $D$ using such line segments, and more specifically, change the integration variable from $\bm{x}$ to $(t,s)$, where 

\begin{equation*}
\bm{x} = t \bm{\lambda} + s \bm{\lambda}^{\perp},
\end{equation*}

\n Here, $t \bm{\lambda}$ varies over the orthogonal projection of $D$ onto the line in direction $\bm{\lambda}$. Translation of $D$ does not affect the nonvanishing of $\tilde{I}$, so we place the projection’s starting point at the origin: $t \in [0, w(\bm{\lambda})]$, where $w(\bm{\lambda})$ is the width of $D$ in the direction $\bm{\lambda}$. Let $D_t$ denote the slice of $D$ perpendicular to $\bm{\lambda}$ at coordinate $t$, i.e.

\begin{equation*}
D_t = \{s : t \bm{\lambda} + s \bm{\lambda}^{\perp} \in D\}
\end{equation*}

\n and let $L(t) = |D_t|$ be the one-dimensional Lebesgue measure of this slice. Thus,

\begin{equation} \label{I tilde slices}
\tilde{I} (\bm{\lambda}) = \int_0^{w(\bm{\lambda})} \int_{D_t} e^{i R t} ds dt = \int_0^{w(\bm{\lambda})} L(t) e^{i R t} dt.
\end{equation}

\n Taking the imaginary part, we arrive at

\begin{equation*}
\Im \tilde{I} (\bm{\lambda}) = \int_0^{w(\bm{\lambda})} L(t) \sin(Rt) dt.
\end{equation*}

\n Because $L(t) \geq 0$ and not identically zero, the integral is nonzero whenever $R w(\bm{\lambda}) \leq \pi$; equivalently, by \eqref{R k M}

\begin{equation*}
k \leq \frac{\pi}{w(\bm{\lambda}) M (\bm{\lambda} \cdot \bm{\eta})}.
\end{equation*}

\n In summary, if the above inequality holds for some direction vector $\bm{\lambda} \in \BS$, then $\tilde{I}(\bm{\lambda}) = I(\bm{\xi})$ is nonzero. Since $\PI = C I$, to obtain the desired contradiction it remains to ensure that the term

\begin{equation*}
C(\bm{\xi}) = \frac{1}{a} - n^2 + \left( \frac{1}{a} - 1 \right) n \, \bm{\xi}\cdot \bm{\eta} 
\end{equation*}

\n is also nonzero. We begin by analyzing this in terms of $\bm{\lambda}$ using \eqref{lambda xi} and  \eqref{R k M}. For convenience, introduce the variable

\begin{equation*}
r = \bm{\lambda} \cdot \bm{\eta} \in [-1, 1].
\end{equation*}

\n Invoking \eqref{lambda xi} and \eqref{R k M} we then have

\begin{equation*}
n \, \bm{\xi}\cdot \bm{\eta} = M(r) r - 1.
\end{equation*}

\n Let us study the equation $C(\bm{\xi}) = 0$. Since this has no solution when $a=1$, we assume $a \neq 1$ and perform straightforward simplifications to rewrite it as

\begin{equation} \label{r M(r)}
r M(r) = a \frac{n^2 - 1}{1-a} =: c_0.
\end{equation}

\n It remains to determine whether the above equation admits a solution $r$ in the interval $[-1,1]$. Using the definition of $M$ from \eqref{M}, we obtain

\begin{equation} \label{c_0/r - r}
\sqrt{n^2 - 1 + r^2} = \frac{c_0}{r} - r.
\end{equation}

\n Squaring both sides and simplifying, we arrive at

\begin{equation*}
r^2 = \frac{c_0^2}{n^2 -1 + 2c_0} = a^2 \frac{n^2 - 1}{1-a^2}.
\end{equation*}

\n For this equation to have a real solution, it is necessary that $a<1$. In that case, the solutions are $r = \pm r_0$, where $r_0$ is the square root of the right-hand side expression (see also \eqref{lambda eta = r}). It is easy to check that when $r=-r_0$ the expression on the right-hand side of \eqref{c_0/r - r} is negative. Therefore, the only solution is $r=r_0>0$. We have $r_0 \in (0,1]$ if and only if $na \leq 1$. Consequently, when $na > 1$, the constant $C(\bm{\xi})$ is always nonzero for any choice of $\bm{\lambda} \in \BS$, which completes the proof of part $(i)$ of Theorem~\ref{THM n>1}. On the other hand, when $na \leq 1$ (in which case it automatically follows that $a<1$) the constant $C(\bm{\xi})$ vanishes if and only if $\bm{\lambda} \cdot \bm{\eta} = r_0$. Let $\bm{\lambda}_\pm$ denote the two (or one) vectors satisfying this equation, then $C(\bm{\xi})$ is nonzero for all $\bm{\lambda} \in \BS \backslash \{\bm{\lambda}_\pm\}$. This completes the proof of part $(ii)$.

\vspace{.1in}

\textbf{The case $n=1$.} In this case $a\neq 1$, because otherwise $a=q=1$. We begin by noting that

\begin{equation*}
C(\bm{\xi}) = \left( \frac{1}{a} - 1 \right) \left( 1  + \bm{\xi} \cdot \bm{\eta} \right).
\end{equation*}

\n When $\bm{\xi}=-\bm{\eta}$ we have $C(\bm{\xi}) =0$ so this choice is not of interest. For $\bm{\xi}\neq -\bm{\eta}$ we have $C(\bm{\xi}) \neq 0$, and we may find $\bm{\lambda} \in \BS$ and a constant $R>0$ such that $k(\bm{\eta} + \bm{\xi}) = R \bm{\lambda}$. A simple calculation shows that $\bm{\lambda}\cdot \bm{\eta}=R/2k>0$. Similarly for any $\bm{\lambda}\in \BS$, with $\bm{\lambda}\cdot \bm{\eta}>0$, we may find $\bm{\xi}\in \BS$, $R>0$, such that  $k(\bm{\eta} + \bm{\xi}) = R \bm{\lambda}$ and $\bm{\xi} \neq -\bm{\eta}$ so $C(\bm{\xi}) \neq 0$. Vectors of interest (those for which $C(\bm{\xi})\neq 0$) are thus those $\bm{\lambda} \in \BS$ satisfying $\bm{\lambda} \cdot \bm{\eta} > 0$.

\n An argument analogous to that for $n>1$ now shows that scattering occurs, provided

\begin{equation*}
k \leq \frac{\pi}{\displaystyle 2 w(\bm{\lambda}) \bm{\lambda} \cdot \bm{\eta}},
\end{equation*}

\n for some $\bm{\lambda} \in \BS$ with $\bm{\lambda} \cdot \bm{\eta} > 0$. Since $D$ is bounded, the widths $w(\bm{\lambda})$ remain bounded. We now choose $\bm{\lambda}$ such that $\bm{\lambda} \cdot \bm{\eta} \to 0^+$, in which case the above inequality grows to cover the entire interval $k \in (0, \infty)$, and we conclude that plane waves always scatter. This finishes the proof of Theorem~\ref{THM n<1}.

\subsection{Proof of Theorem~\ref{THM convex}}

Our starting point is the formula \eqref{I tilde slices}:

\begin{equation*}
I(\bm{\xi}) = \tilde{I} (\bm{\lambda}) = \int_0^{w(\bm{\lambda})} L(t) e^{i R(\bm{\lambda}) t} dt,
\end{equation*}

\n where $\bm{\lambda} \in \BS$ is an arbitrary vector and its relation to $\bm{\xi}$ is given by \eqref{lambda xi}. As before, $R(\bm{\lambda}) = k M (\bm{\lambda} \cdot \bm{\eta})$, and $L(t)$ denotes the length of the slice of $D$ perpendicular to $\bm{\lambda}$ at coordinate $t \in [0,w(\bm{\lambda})]$. Since $D$ is convex, the function $L(t)$ is concave. This can be easily seen once we rotate the coordinate axes so that $\bm{\lambda}$ lies along the horizontal direction. Then we can write $L(t) = L_2(t) - L_1(t)$, where $s=L_1(t)$ and $s=L_2(t)$ denote the lower and upper portions of the boundary $\partial D$, respectively. Both $L_2$ and $-L_1$ are concave functions. 

Taking the real part of the above expression, suppressing the $\bm{\lambda}$ dependence in the notation, and integrating by parts, we obtain\footnote{$L'$ exists and is an integrable function, since $D$ is convex.}

\begin{equation*}
\Re \tilde{I}  = \int_0^{w} L(t) \cos \left( R t \right) dt = - \frac{1}{R} \int_0^{w} L'(t) \sin \left( R t \right) dt.
\end{equation*}

\n Here, the boundary terms vanish because $L(0) = L(w) = 0$. This is due to the strict convexity of $D$, which prevents its boundary from containing line segments. Changing variables reduces the integral to

\begin{equation*}
\Re \tilde{I}  = \int_0^{wR} \tilde{L}(x) \sin x dx, \qquad \qquad \tilde{L}(x) = -\frac{1}{R^2} L'\left( \frac{x}{R} \right).
\end{equation*}

\n Since $L$ is concave, $\tilde{L}$ is an increasing function. Suppose now that the parameters $w$ and $R$ are such that $wR = 2\pi$. In that case,

\begin{equation*}
\Re \tilde{I}  = \int_0^{2 \pi} \tilde{L}(x) \sin x dx = \int_0^\pi + \int_\pi^{2\pi} = \int_0^\pi \left[ \tilde{L}(x) - \tilde{L}(\pi + x) \right] \sin x dx  < 0.
\end{equation*}

\n In other words, $\tilde{L}$ puts more weight on the second interval $(\pi, 2\pi)$ on which sine is negative, making the total integral negative. It is easy to see that the same conclusion can be obtained also in the case where $w R = 2\pi m$ for some integer $m=1,2,...$ (simply split the integral into a sum of $m$ terms and conclude that each term is negative). Writing more explicitly, we obtain that at the wave numbers

\begin{equation*}
k = \frac{2\pi m}{w(\bm{\lambda}) M(\bm{\lambda} \cdot \bm{\eta})}, \qquad \qquad \bm{\lambda} \in \BS, \ m=1,2,...
\end{equation*}

\n the integral $\tilde{I}(\bm{\lambda})$ is nonzero. The right-hand side expression is a continuous function of $\bm{\lambda}$ and as $\bm{\lambda}$ varies over $\BS$, it attains all values between its minimum and maximum. In other words, $\tilde{I}(\bm{\lambda})$ is nonzero for some $\bm{\lambda} \in \BS$ (and thus $I$ is nonzero for some $\bm{\xi}\in \BS$) provided 

\begin{equation*}
k \in \left[ \frac{2\pi m}{h_1(\bm{\eta})}, \frac{2\pi m}{h_0(\bm{\eta})}\right] \qquad \qquad m=1,2,...
\end{equation*}

\n where $h_0$ and $h_1$ are defined in \eqref{h_0 def} and \eqref{h_1 def}, respectively. To conclude the proof, we recall that $C(\bm{\xi})$ is nonvanishing for all $\bm{\lambda}$ when $na > 1$, whereas for $na \leq 1$ we must exclude the directions $\bm{\lambda} = \bm{\lambda}_\pm$ (see part $(ii)$ of Theorem~\ref{THM n>1}).

\subsection{Proof of Lemma~\ref{LEM phi}} \label{SECT LEM}

Setting $u = u^{\text{tr}} - u^{\text{in}}$ and using \eqref{u^in Helm} we rewrite the system \eqref{nonsc system} in terms of $u$ and $u^{\text{in}}$:

\begin{equation*}
\begin{cases}
\Delta u + k^2 n^2 u = k^2 \left(1 - n^2 \right) u^{\text{in}} \qquad  & \text{in} \ D,
\\[.1in]
u = 0, \qquad \partial_\nu u = \left( \frac{1}{a} - 1 \right) \partial_\nu u^{\text{in}} &\text{on} \ \Gamma = \partial D,
\end{cases}
\end{equation*}

\n Multiplying the above PDE by a test function $\phi \in H^2(D)$ and integrating over $D$, we obtain

\begin{equation*}
\begin{split}
k^2 \left(1 - n^2 \right) \int_D u^{\text{in}} \phi d\bm{x} &= \int_D \Delta u \phi + k^2 n^2 u \phi d\bm{x} = 
\\[.1in]
&=
\int_D u \Delta \phi d\bm{x} + \int_{\Gamma} \left(\partial_\nu u \phi - u \partial_\nu \phi\right) ds + \int_D k^2 n^2 u \phi d\bm{x},    
\end{split}
\end{equation*}

\n where we have used  Green's identity. Now, using the boundary conditions satisfied by $u$, we get

\begin{equation*}
k^2 \left(1 - n^2 \right) \int_D u^{\text{in}} \phi d\bm{x} = \int_D u \left( \Delta \phi + k^2 n^2 \phi \right) d\bm{x} + \left( \frac{1}{a} - 1 \right) \int_{\Gamma} \partial_\nu u^{\text{in}} \, \phi ds.
\end{equation*}

\n Imposing $\Delta \phi + k^2 n^2 \phi = 0$ in $D$, the above formula reduces to

\begin{equation} \label{1}
k^2 \left(1 - n^2 \right) \int_D u^{\text{in}} \phi d\bm{x} + \left( 1 - \frac{1}{a} \right) \int_{\Gamma} \partial_\nu u^{\text{in}} \, \phi ds = 0.
\end{equation}

\n We use the divergence theorem to rewrite the second integral above:

\begin{eqnarray*}
\int_{\Gamma} \partial_\nu u^{\text{in}} \, \phi ds &=& \int_D \div \left( \nabla  u^{\text{in}} \, \phi \right) d\bm{x} = \int_D \left(\nabla u^{\text{in}} \cdot \nabla \phi + \Delta u^{\text{in}} \, \phi\right) d\bm{x} \\
 &=& \int_D \left(\nabla u^{\text{in}} \cdot \nabla \phi - k^2 u^{\text{in}} \phi\right) d\bm{x}. 
\end{eqnarray*}

\n The last step follows from the PDE \eqref{u^in Helm} satisfied by the incident wave. To complete the proof, we substitute the above expression into \eqref{1}.

\section*{Acknowledgements}
{The authors would like to thank Fioralba Cakoni for many interesting discussions on non-scattering and transmission eigenvalues. The first author would also like to thank Andrea Al{\`u} for introducing him to the Salisbury screen model. This work was partially supported by NSF Grant DMS-22-05912 }

\bibliographystyle{abbrv}
\bibliography{ref}

\end{document}